\documentclass[11pt]{amsart}
\usepackage[utf8]{inputenc}
\usepackage[english]{babel}
 
\usepackage{amsthm}
\usepackage{amsfonts}
\usepackage{amssymb}
\usepackage{bbold}
\usepackage{tikz}
\usetikzlibrary{cd}

\DeclareMathOperator{\cl}{\sf cl}

\newcommand{\Z}{\ensuremath{\mathbf Z}}

\renewcommand{\phi}{\varphi}

\newcommand{\tc}{{\sf {TC}}}
\newcommand{\cat}{{{\sf {cat}}}}
\renewcommand{\k}{{\sf k}}

\DeclareMathOperator{\zcl}{\sf zcl}

\newtheorem{theorem}{Theorem}
\newtheorem{proposition}{Proposition}
\newtheorem{lemma}{Lemma}
\newtheorem{corollary}{Corollary}
\theoremstyle{definition}
\newtheorem{definition}{Definition}
\newtheorem{question}{Question}
\begin{document}
\title{Generating functions and topological complexity}
\author{Michael Farber}
\address{School of Mathematical Sciences \\
Queen Mary University of London\\
London, E1 4NS, 
United Kingdom}
\email{m.farber@qmul.ac.uk}
\thanks{M. Farber was partially supported by a grant from the Leverhulme Foundation}

\author {Daisuke Kishimoto}
\address{Department of Mathematics\\
 Graduate School of Science\\
Kyoto University, Kyoto 606-8502, Japan}
\email{kishi@math.kyoto-u.ac.jp}

\author {Donald Stanley}
\address{Department of Mathematics and Statistics,
University of Regina \\
3737 Wascana Parkway,
Regina, Saskatchewan,  Canada}
\email{Donald.Stanley@uregina.ca}

\thanks{The authors thank Fields Institute for Research in Mathematical Sciences for hospitality} 
\maketitle

\begin{abstract} 
We examine the rationality conjecture raised in \cite{FO} which states that (a) the formal power series 
$\sum_{r\ge 1} \tc_{r+1}(X)\cdot x^r$ represents a rational function of $x$ with a single pole of order 2 at $x=1$ and (b) the leading coefficient of the pole equals $\cat(X)$. Here $X$ is a finite CW-complex and for $r\ge 2$ the symbol $\tc_r(X)$ denotes its $r$-th sequential topological complexity. 
We analyse an example (violating the Ganea conjecture) and conclude that part (b) of the rationality conjecture is false in general. Besides, 
we establish a cohomological version of the rationality conjecture.
\end{abstract}
\section{Introduction}

For a topological space $X$ the symbol $X^r$ stands for the $r$-th Cartesian power of $X$ and $\Delta_r\colon X\rightarrow X^r$
denotes the diagonal map $\Delta_r(x) =(x, x, \dots, x)$, where $x\in X$. 
For $r\ge 2$ the {\it $r$-th topological complexity of $X$} is an integer $\tc_r(X)\ge 0$ defined as 
the sectional category of the  diagonal map 
$\Delta_r: X\to X^r$. In other words, $\tc_r(X)$ is the smallest $k\ge 0$ such that there exists an open cover 
$X^r=U_0\cup U_1\cup \dots\cup U_k$ with the property that each set $U_i$ admits a homotopy section $s_i:U_i\to X$ of $\Delta_r$, 
 i.e. the composition $\Delta_r\circ s_i$ is homotopic to the inclusion $U_i\to X^r$ for $i=0, 1, \dots, k$. The above definition is equivalent to the one given in \cite{R}. For $r>2$ the number $\tc_r(X)$ is sometimes called {\it the higher or sequential topological complexity}. 
$\tc_r(X)$ is a homotopy invariant of $X$. 
 For a mechanical system having $X$ as its configuration space, $\tc_r(X)$ can be interpreted as complexity of motion planning 
 algorithms taking $r$ states $(x_1, x_2, \dots, x_r)\in X^r$ as input and producing a continuous motion of each of the points $x_i$ moving to 
 a {\it consensus state} $y\in X$. See \cite{R},  \cite{GGY}, \cite{FO} for more detail.

Putting all the numbers  $\tc_r(X)$ together, we obtain a formal power series
\begin{eqnarray}\label{ftc}
\mathcal F_X(x)=\sum_{r\ge 1} \tc_{r+1}(X)\, x^r,\end{eqnarray}
{\em the $\tc$-generating function} of $X$.
It was observed in \cite{FO} that in many examples this series represents a rational function of a special type.
Based on these computations, it was conjectured in \cite{FO} that {\it for any finite CW-complex $X$ the power series $\mathcal F_X(x)$ is a rational function of the form 
\begin{eqnarray}\label{fx}
\frac{P_X(x)}{(1-x)^2}
\end{eqnarray}
where the numerator $P_X(x)$ is an integer polynomial satisfying $P_X(1)=\cat(X)$. } The conjecture is true for all spheres, closed orientable surfaces, simply connected symplectic manifolds and also for the Eilenberg-MacLane spaces $X=K(H_\Gamma, 1)$ where $H_\Gamma$ is the right angled Artin group determined by a graph $\Gamma$, see \cite{FO}. 

The goal of this paper is to show that the above conjecture is false in the form it was stated in \cite{FO}, namely we show that for a specific finite CW-complex $X$ the power series $\mathcal F_X(x)$ is a rational function of the form (\ref{fx}) however the value $P_X(1)$ is distinct from 
$\cat(X)$. 

\section{Rationality of homological power series}

\subsection{} We start with a general observation:

\begin{lemma}\label{lm3}
Consider a formal power series $\mathcal F(x) = \sum_{r\ge 0} t_r x^r$ with integer coefficients $t_r\in \Z$. 
The following conditions are equivalent:

\begin{enumerate}
\item[{\rm (a)}] $\mathcal F(x)$ represents a rational function of the form 
$P(x)\cdot(1-x)^{-2}$
where $P(x)$ is an integer polynomial;
\item[{\rm (b)}] the coefficients $t_r$ satisfy 
for all $r$ large enough
a recurrent equation of the form
$t_{r}=t_{r-1}+a;$
\item[{\rm (c)}]  for some $a$ and $c$ one has 
$t_{r-1} +a\le t_{r}\le ra + c$
for all $r$ large enough.
\end{enumerate}

\end{lemma}
\begin{proof} 
First note that $\sum_{r\ge 0} x^r=(1-x)^{-1}$ and $\sum_{r\ge 0} (r+1)x^r=(1-x)^{-2}$. 

Assuming (a), one can write the polynomial $P(x)$ in the form $P(x) = a_0+a_1(1-x) +(1-x)^2 q(x)$ which leads to  
$\mathcal F(x)= \frac{a_0}{(1-x)^2} + \frac{a_1}{(1-x)} + q(x),$
where $q(x)$ is an integer polynomial. Equating coefficients, for $r$ large enough we obtain $t_r=a_0(r+1) +a_1$ and hence $t_{r}=t_{r-1}+a$ with $a=a_0$. Hence (a) $\implies$ (b).

Obviously (b)$\implies$(c). 

Given (c), we may write $t_r=ra +s_r$ and the LHS inequality in (c) gives $s_{r-1}\le s_r$ while the RHS inequality gives $s_r\le c$, for $r$ large. Hence the integer sequence $s_r$ is bounded above and eventually increasing, thus it converges, and therefore it is eventually constant, i.e. $s_r=d$ and $t_r=ra +d$ for all $r$ large enough, implying (b). 

Finally, to show that (b)$\implies$(a) we note that $t_r=ra +d$ for all large $r$ implies that $\mathcal F(x)$ equals the sum of 
$\frac{ax}{(1-x)^2} + \frac{d}{1-x}$ and an integer polynomial and  therefore $\mathcal F(x)$ can be written in form $P(x)\cdot (1-x)^{-2}$ where $P(x)$ is an integer polynomial; note that $P(1)=a$. 
%
%
%
\end{proof}

\subsection{The cup-length} Let $A$ be a graded commutative finite dimensional algebra with unit over a field $\k$. We shall assume that $A$ is {\it connected}, i.e. the degree zero part $A^0=\k$ is generated by the unit $1\in A^0$. A typical example is the case when $A=H^\ast(X, \k)$ where $X$ is a connected finite CW-complex. 

We denote by $\cl(A)$ {\it the cup-length of $A$}, it is the largest number of elements of $A$ of positive degree with nonzero product. 
Equivalently, $\cl(A)$ equals the largest number of homogeneous elements of $A$ of positive degree with nonzero product. 

If $B$ is another such algebra one may consider the tensor product $A\otimes B$, the algebra structure is given by 
$(a\otimes b)\cdot (a'\otimes b')= (-1)^{|a'|\cdot |b|} aa' \otimes bb'$. Here $a, a', b, b'$ are homogeneous elements and the symbol $|a|$ denotes the degree of $a$. 

One notes that
\begin{eqnarray}\label{ab}
\cl(A\otimes B) = \cl(A) +\cl (B).
\end{eqnarray}
Indeed, if $a_1, \dots, a_r\in A$ and $b_1, \dots, b_s\in B$ elements of positive degree with $a_1 a_2\dots a_r\not=0\in A$ and 
$b_1 b_2\dots b_s\not=0\in B$ then $$\prod_{i=1}^r a_i\otimes 1\cdot \prod_{j=1}^s 1\otimes b_j \not=0\in A\otimes B.$$
Conversely, to show that $\cl(A\otimes B) \le \cl(A) +\cl(B)$ consider a nonzero product of homogeneous tensors 
$$\prod_{\ell=1}^{\cl(A\otimes B)} a_\ell\otimes b_\ell\, \not= 0\, \in \, A\otimes B.$$
The set of indices $\{1, \dots, \cl(A\otimes B)\}=C_1\sqcup C_2\sqcup C_3$ can be split into 3 disjoint sets where for $\ell\in C_1$ one has $|b_\ell|=0$ and for $\ell\in C_2$ one has $|a_\ell|=0$ and for $\ell\in C_3$ one has $|a_\ell|\not=0\not= |b_\ell|$. Clearly, 
$|C_1|+|C_3|\le \cl(A)$ and $|C_2|+|C_3|\le \cl(B)$ implying $\cl(A\otimes B) = |C_1|+|C_2|+|C_3|\le \cl(A)+\cl(B)$. 

For an integer $r\ge 0$ we denote by $A^r$ the $r$-fold tensor product $$A^r= A\otimes A\otimes \cdots \otimes A\quad \mbox{($r$ factors)}.$$ 
The above discussion gives:
\begin{corollary}\label{cor1} For any $r>0$ one has
$\cl(A^r) = r\cl(A),$ and therefore 
$$\sum_{r\ge 1}\cl(A^r)\cdot x^r=\frac{\cl(A)\cdot x}{(1-x)^2}.$$
\end{corollary}

\subsection{Zero-divisors-cup-length} As above, let $A$ be a connected graded commutative finite dimensional algebra with unit over a field $k$. 
For any $r\ge 2$ one has the multiplication map
\begin{eqnarray}
\mu_r: A^r \to A \quad \mbox{where}\quad \mu_r(a_1\otimes a_2\otimes \dots\otimes a_r) = a_1 a_2 \cdots a_r.
\end{eqnarray}
This map is a homomorphism of algebras and its kernel has a graded algebra structure. 
The kernel $\ker(\mu_r)$ is called {\it the ideal of zero divisors}. 

\begin{definition}
{\it The $r$-th zero-divisors-cup-length of $A$}, denoted $\zcl_r(A)$, is the longest nontrivial product of elements of $\ker(\mu_r)$. 
\end{definition} 

\begin{lemma}\label{prop1}
The power series $\sum_{r\ge 1} \zcl_{r+1}(A)\cdot x^r$ is a rational function of the form $P(x) \cdot (1-x)^{-2}$ 
where $P(x)$ is an integer polynomial satisfying $P(1)=\cl(A)$. 
\end{lemma}
\begin{proof} First we observe that (using (\ref{ab})),
\begin{eqnarray}\label{one}
\zcl_r(A)\le \cl(A^r) = r\cdot \cl(A).\end{eqnarray}
Next we claim that 
\begin{eqnarray}\label{two}
\zcl_{r+1}(A) \ge \zcl_r(A) +\cl(A).
\end{eqnarray}
Indeed, suppose that zero-divisors $x_1, \dots, x_\ell\in \ker(\mu_r)$ are such that their product $x_1\cdot x_2\cdots x_\ell\not=0\in \ker(\mu_r)$ is nonzero; here $\ell=\zcl_r(A)$. If $\ell'$ denotes $\cl(A)$, consider elements of positive degree 
$y_1, \dots, y_{\ell'}\in A$ with nonzero product $y_1\cdot y_2\cdots y_{\ell'}\not=0\in A$. For $i=1, 2, \dots, \ell'$ define the elements 
$\overline y_i\in A^{r+1}$ given by 
$\overline y_i=1\otimes 1\otimes \cdots 1\otimes y_i- 
y_i\otimes 1\otimes \cdots\otimes 1.$ Clearly, $\overline y_i$ lies in $\ker(\mu_{r+1})$. 
Besides, let $\overline x_j=x_j\otimes 1\in A^{r+1}$; again, we have $\overline x_j\in \ker(\mu_{r+1})$.

Finally we claim that the product 
$$\Pi= \overline x_1\cdot \overline x_2\cdots \overline x_\ell \cdot \overline y_1\cdot \overline y_2\cdots \overline y_{\ell'}\in A^{r+1}$$
is nonzero. Let $\psi: A\to k$ be a linear functional satisfying $\psi(y_1 y_2\dots y_{\ell'})=1$ and $\psi|{A^i}=0$ for all $i$ distinct from 
the degree of the product $y_1 y_2\dots y_{\ell'}$.  
Then the map $1\otimes 1\otimes \dots\otimes 1\otimes \psi: A^{r+1}\to A^r$ satisfies 
$(1\otimes 1\otimes \dots\otimes 1\otimes \psi)(\Pi) =x_1\cdot x_2\cdots x_\ell\not=0$ showing that $\Pi\not=0$. 
This proves (\ref{two}). 

The statement of Lemma \ref{prop1} now follows (by applying Lemma \ref{lm3})
from (\ref{one}) and (\ref{two}). \end{proof}

In applications we have $A=H^\ast(X;\k)$ is the cohomology algebra of a connected finite CW-complex $X$; then the number
$\zcl_r(A)$ serves as the lower bound for $\tc_r(X)$, see \cite{R}. 

We shall abbreviate the notation $\zcl_r(H^\ast(X,\k))$ to $\zcl_r(X;\k)$ or to $\zcl_r(X)$ assuming that the field $\k$ is specified.

\begin{corollary}\label{cor2}
Let $X$ be a connected finite CW-complex. For any field $\k$, the formal power series 
$\sum_{r\ge 1}\zcl_{r+1} (X;\k)\cdot x^r$
represents a rational function of the form $P(x)\cdot (1-x)^{-2}$ with $P(1) = \cl(H^\ast(X;\k))$.
\end{corollary}

\begin{theorem}\label{cor3}
Let $X$ be a connected finite CW-complex such that for all large $r$ one has $\tc_r(X)=\zcl_r(H^\ast(X;\k))$ where $\k$ is a field. 
Then the $\tc$-generating function (\ref{ftc}) is a rational function of form $P(x)\cdot (1-x)^{-2}$ with $P(1) = \cl(H^\ast(X;\k))$. 
\end{theorem}

\subsection{A counterexample} Theorem \ref{cor3} suggests how to produce an example contradicting the rationality conjecture as stated in \cite{FO}. 
Namely, suppose that 
$X$ is a finite CW-complex satisfying $\tc_r(X)=\zcl_r(X;\k)$ for all large $r$ although $\cl(H^\ast(X;\k))< \cat(X)$. Then the $\tc$-generating function (\ref{ftc}) is a rational function of form $P(x)\cdot (1-x)^{-2}$ while the leading term $P(1)$ {\it is smaller than}  $ \cat(X).$

In section 5 of \cite{Sta} the author constructs a remarkable finite CW-complex $X$. In general $X$ depends on a choice of a prime $p$ but
for our purposes in this paper we shall for simplicity assume that $p=3$. The complex $X$ has 3 cells of dimensions 2, 3, 11. Its main properties are:
\begin{enumerate}
\item $\cat(X)=2$;
\item $\cat(X\times X)=2$;
\item For any non-contractible space $Y$ one has $$\cat(X\times Y)< \cat(X)+\cat(Y).$$ 
\end{enumerate}
In particular $\cat(X\times S^n) = \cat(X)=2$ for any $n\ge 1$, i.e.
$X$ is a counterexample to the Ganea conjecture. 
Using these properties we find by induction that $\cat(X^r) \le r$ for any $r\ge 2$. 

The cohomology algebra 
$A=H^\ast(X;\k)$
of $X$ has generators $a_2, a_3, a_{11}$ of degrees $2, 3, 11$ correspondingly; all pairwise products of the generators vanish for dimensional reasons. 
We see that the cup-length $$\cl(A) =\cl(H^\ast(X;\k))= 1$$ and using Corollary \ref{cor1} we obtain $\cl(H^\ast(X^r;\k))=r$ for any $r$. 
Thus combining with information given above, we see that for any $r\ge 2$ one has
\begin{eqnarray}\label{onehas}
\cat(X^r)=r. 
\end{eqnarray}

Next we show that $\zcl_r(A)\ge r$ for any $r\ge 2$, where $A=H^\ast(X;\k)$. 
Consider the class 
$x=a_2\otimes 1- 1\otimes a_2\in \ker(\mu_2)\subset A^2.$
We see that $x^2= -2 a_2\otimes a_2\not=0$ which implies $\zcl_2(A)\ge 2$. 
Applying (\ref{two}) inductively we deduce $\zcl_r(A)\ge r$ for any $r\ge 2$.
Therefore, 
%
\begin{eqnarray}\label{onehas2}
\tc_r(X) \ge r, \quad \mbox{for any} \, r\ge 2.
\end{eqnarray}
Combining (\ref{onehas}) and (\ref{onehas2}) with the well-known inequalities $\zcl_r(X) \le \tc_r(X)\le \cat(X^r)$ we obtain
\begin{eqnarray}
\tc_r(X)=r =\zcl_r(X)\quad \mbox{for any} \, r\ge 2.
\end{eqnarray}
By Theorem \ref{cor3} the $\tc$-generating function $\sum_{r=1} \tc_{r+1}(X) \cdot x^r$ is a rational function of the form 
$P(x)\cdot (1-x)^{-2}$ where $P(x)$ is an integer polynomial with $P(1)=1$. Here $1=\cl(A)< \cat(X)$. Thus this example violates the original rationality conjecture of \cite{FO}. 

\subsection{Conclusion} The discussion above suggests a weaker version of the rationality conjecture which can be stated as follows: {\it for any finite CW-complex $X$ the $\tc$-generating function $\mathcal F_X(x) = \sum_{r\ge 1} \tc_{r+1}(X)\cdot x^r$ is a rational function with a single pole of order 2 at $x=1$}.


\begin{thebibliography}{99}

\bibitem{FO} M. Farber, J. Oprea, \textit{Higher topological complexity of aspherical spaces}. Topology Appl. {\bf 258} (2019), 142–160. 

\bibitem{GGY} J. Gonzalez, B. Gutierrez, S. Yuzvinsky
\textit{Higher topological complexity of subcomplexes of products of spheres and related polyhedral product spaces},
Topol. Methods Nonlinear Anal., {\bf 48} (2) (2016), pp. 419-451

\bibitem{R} Y. Rudyak,
\textit{On higher analogs of topological complexity. }
Topology Appl. {\bf 157} (2010), no. 5, 916–920. 

\bibitem{Sta} D. Stanley, \textit{On the Lusternik-Schnirelmann Category
of Maps}, Canad. J. Math. Vol. {\bf 54} (3), 2002 pp. 608–633.
\end{thebibliography}
\end{document}